\documentclass[11pt,a4paper]{amsart}
\usepackage{amsmath}
\usepackage{amssymb}
\usepackage{amsthm}
\usepackage{amsfonts}
\usepackage{mathrsfs}

\newtheorem{thm}{Theorem}[section]
\newtheorem{lemma}[thm]{Lemma}
\newtheorem{coro}[thm]{Corollary}
\newtheorem{prop}[thm]{Proposition}

\newenvironment{remark}{{\vskip3pt\noindent\bf Remark.}}{\vskip6pt}

\newcommand{\diam}{\operatorname{diam}}

\newcommand{\dist}{\operatorname{dist}}

\newcommand{\loc}{\operatorname{loc}}

\newcommand{\Lip}{\operatorname{Lip}}

\newcommand{\bdry}{\partial}
\newcommand{\R}{\mathbb R}

\newcommand{\N}{\mathbb N}

\newcommand{\Ch}{\mathcal C}

\newcommand{\Ha}{\mathcal H}
\newcommand{\Hh}{\mathscr H}

\newcommand{\sub}{\subset}

\newcommand{\ol}{\overline}

\newcommand{\Char}[1]{\chi_{\lower 1.5pt\hbox{$\scriptscriptstyle #1$}}}
\newcommand{\dom}{{d_{\Omega}}}

\newcommand{\vint}[1]{{\mathchoice % muu versio
          {\mathop{\vrule width 6pt height 3 pt depth -2.5pt
                  \kern -8pt \intop}\nolimits}%
          {\mathop{\vrule width 5pt height 3 pt depth -2.6pt
                  \kern -6pt \intop}\nolimits}%
          {\mathop{\vrule width 5pt height 3 pt depth -2.6pt
                  \kern -6pt \intop}\nolimits}%
          {\mathop{\vrule width 5pt height 3 pt depth -2.6pt
                  \kern -6pt \intop}\nolimits}}_{\!\!\!\! #1}}

\begin{document}

\title[Hardy inequalities beyond Lipschitz domains]{Weighted Hardy inequalities\\ beyond Lipschitz domains}
\author{Juha Lehrb\"ack} 

\thanks{
The author was supported in part by
the Academy of Finland, grant no.\ 120972}

\address{Department of Mathematics and Statistics, 
P.O. Box 35 (MaD), 
FIN-40014 University of Jyv\"askyl\"a, 
Finland}
\email{\tt juha.lehrback@jyu.fi}
%\keywords{}

\begin{abstract}

It is a well-known fact that in a Lipschitz domain $\Omega\sub\R^n$
a $p$-Hardy inequality, with weight $\dist(x,\bdry\Omega)^\beta$,
holds for all $u\in C_0^\infty(\Omega)$ whenever $\beta<p-1$. 
We show that actually the same is true under the sole assumption that the
boundary of the domain satisfies a uniform density condition
with the exponent $\lambda=n-1$. Corresponding results also hold
for smaller exponents, and, in fact, 
our methods work in general metric spaces
satisfying standard structural assumptions. 

\end{abstract}

\maketitle

\section{Introduction}

We say that
an open set $\Omega\sub\R^n$
admits the $(p,\beta)$-Hardy inequality,
for $1<p<\infty$ and $\beta\in\R$, if there
exists a constant $C>0$ such that the inequality
\begin{equation}\label{eq: Hardy}
\int_{\Omega} |u(x)|^p\, \dom(x)^{\beta-p}\,dx
   \leq C\int_{\Omega} |\nabla u(x)|^p\, \dom(x)^\beta \, dx,
\end{equation}
where $\dom(x)=\dist(x,\bdry\Omega)$,
holds for every $u\in C_0^\infty(\Omega)$.
After the one-dimensional considerations
by G.\,H.\,Hardy et.\ al.\ in the early 20th century 
(see \cite[\textsection 330]{HLP} and the
references therein), these inequlities were
introduced in dimension $n\geq 2$ by J.\,Ne\v cas.
The main point of reflection for our studies and results
is his theorem from \cite{necas} (see also A.\,Kufner \cite{kuf} for
this and related results):

\begin{thm}[Ne\v cas 1962]\label{thm: necas}
Let $\Omega\sub\R^n$ be a bounded Lipschitz domain
and let $1<p<\infty$. Then $\Omega$ admits
the $(p,\beta)$-Hardy inequality for all $\beta<p-1$.
\end{thm}

Recall that a domain (an open and connected set) is said to be a Lipschitz domain
if the boundary $\bdry\Omega$ can be represented locally as 
graphs of Lipschitz continuous functions. It follows
from this definition that the boundary of a Lipschitz domain $\Omega\sub\R^n$
is both `smooth' and `thick', the latter
for instance in the sense that
\begin{equation}\label{eq: n-1 estimate}
 \Ha_\infty^{n-1}\big(\bdry\Omega\cap B(x,2\dom(x))\big)\geq C_0 \dom(x)^{n-1}
\end{equation}
for all $x\in\Omega$, where $\Ha_\infty^{\lambda}$ denotes the
$\lambda$-dimensional Hausdorff content. 
Our main result is the following far-reaching generalization of
Theorem~\ref{thm: necas}. 
\begin{thm}\label{thm: main}
Let $\Omega\sub\R^n$ be an open set and let $1<p<\infty$.
Suppose that there exist
an exponent $0\leq\lambda\leq n-1$ and a constant $C_0>0$ such that
\begin{equation}\label{eq: euc main estimate}
 \Ha_\infty^\lambda\big(\bdry\Omega\cap B(x,2\dom(x))\big)\geq C_0 \dom(x)^\lambda
\end{equation}
for all $x\in\Omega$.
Then $\Omega$ admits the $(p,\beta)$-Hardy inequality for all
$\beta<p-n+\lambda$.
\end{thm}

We note that Theorem~\ref{thm: main} was partly conjectured in \cite{kole}.
By \eqref{eq: n-1 estimate}, Theorem~\ref{thm: necas} follows from
Theorem~\ref{thm: main} by taking $\lambda=n-1$. 
We conclude that the smoothness of
a Lipschitz boundary plays no role here, as the thickness alone suffices
for Hardy inequalities. 
Another important and interesting consequence of Theorem~\ref{thm: main}
is that each simply connected domain in $\R^2$ admits the
$(p,\beta)$-Hardy (at least) for all $\beta<p-1$.

The bound $\beta<p-n+\lambda$ in Theorem~\ref{thm: main} is optimal.
In fact, it was shown in~\cite{lerest} (following the unweighted considerations
from~\cite{KZ}) that if $\Omega$ has an isolated boundary part of 
(Hausdorff) dimension $\lambda$, then it is not possible for $\Omega$ to
admit the $(p,p-n+\lambda)$-Hardy inequality, although 
the $(p,\beta)$-Hardy inequality might still hold for some larger $\beta$; see~\cite{lerest}.
Also the bound $\lambda\leq n-1$ (whence $\beta<p-1$) is essential,
as examples from \cite{kole} show.

Conditions of the type~\eqref{eq: euc main estimate} are referred to as 
`inner boundary density conditions'. By \cite[Thm 1]{lepw},
such conditions are actually equivalent to similar density conditions
for the complement $\Omega^c$. In particular, \eqref{eq: euc main estimate}
holds for all $x\in\Omega$ with an exponent $\lambda>n-q$
if and only if $\Omega^c$ is \emph{uniformly $q$-fat} (see e.g.~\cite{lepw}
for the definition). It follows that we  
can rewrite Theorem~\ref{thm: main} in the spirit of the unweighted
results by
A.\,Ancona~\cite{ancona} (the case $p=2$) and J.\,Lewis~\cite{lewis}, 
and generalize the weighted inequalities of A.\,Wannebo~\cite{W},
as follows:

\begin{coro}\label{coro: fatti}
Let $\Omega\sub\R^n$ be an open set and assume
that $\R^n\setminus\Omega$ is uniformly $q$-fat for all $q>s\geq 1$. 
Then $\Omega$ admits the
$(p,\beta)$-Hardy inequality whenever $1<p<\infty$ and $\beta < p-s$.
\end{coro}

To be precise, Wannebo proved in~\cite{W} that uniform $p$-fatness of the complement, for $1<p<\infty$, 
suffices for $(p,\beta)$-Hardy inequalities for all $\beta<\beta_0$, where
$\beta_0$ is some small positive number. Hence, the main novelties in 
Corollary~\ref{coro: fatti} are
that we get an explicit and sharp bound for such an $\beta_0$ and that we can also
deal with the cases when $p\leq s$, where $s$ is the `optimal' fatness
of the complement; of course, in such cases we must have $\beta < 0$.

Our proof of Theorem~\ref{thm: main} is based on rather standard 
`geometric' ideas
and methods, which actually work in the much more general
setting of a metric measure space, provided that the space satisfies some structural conditions;
see Section~\ref{sect: mms} for this general framework. 
In the case $\beta\leq 0$, Theorem~\ref{thm: main} is 
an immediate consequence of
a stronger (and for $\beta<0$ previously unknown) result concerning 
the so-called pointwise Hardy inequalities;
see Section~\ref{sect: hardy} for the definition and Theorem~\ref{thm: pw for beta<0}
for the result. In fact, Theorem~\ref{thm: pw for beta<0} together with \cite[Thm 3.1]{lene}
shows that, for $\beta\leq 0$, condition~\eqref{eq: euc main estimate} for some $\lambda>n-p+\beta$
is actually both necessary and sufficient for a domain $\Omega\sub\R^n$ to admit a
pointwise version of the $(p,\beta)$-Hardy inequality (Corollary~\ref{coro: char}).

On the other hand, for $\beta>0$ density condition~\eqref{eq: euc main estimate}
with an exponent $n-p+\beta < \lambda\leq n$
is still necessary for pointwise inequalities by~\cite{lene}, but
not anymore sufficient even for the usual Hardy inequality~\eqref{eq: Hardy}, as was shown by
examples in~\cite{kole}. Nevertheless, if $0<\beta<p-n+\lambda\leq p-1$
and \eqref{eq: euc main estimate} holds, 
then the $(p,\beta)$-Hardy inequality can be obtained from the unweighted
$(p-\beta)$-Hardy inequality by an integration trick, and the
theorem follows; see Section~\ref{sect: proofs}
for the details. 

In \cite{kole} it was actually shown that~\eqref{eq: euc main estimate} together with an additional
accessibility condition suffices for a pointwise $(p,\beta)$-Hardy inequality
for all $\beta<p-n+\lambda$. By our Theorem~\ref{thm: pw for beta<0},
such an accessibility condition can now be dropped altogether if $\beta\leq 0$,
and, by Theorem~\ref{thm: main}, for $0<\beta<p-1$ if we are only interested in the validity of
the integral Hardy inequality \eqref{eq: Hardy}.
The main theorem of \cite{kole} concerning 
pointwise inequalities for $\beta>0$, with accessibility, is generalized to metric spaces in 
Theorem~\ref{thm: pw for beta>0} with a simplified proof.
As the examples from~\cite{kole} did show that for $\beta\geq p-1$ 
density condition~\eqref{eq: euc main estimate} alone is not sufficient
for the pointwise $(p,\beta)$-Hardy inequality, we conclude that the 
only piece that is still missing from the complete picture 
is whether the assumptions of Theorem~\ref{thm: main} 
always (that is, also for $0<\beta<p-n+\lambda\leq p-1$) suffice for a
pointwise version of the $(p,\beta)$-Hardy inequality.
We conjecture that this is the case, and mention, for the record, that
this question was really the essence of the (now proven)
Conjecture 1.6 of~\cite{kole}.

The organization of this paper is as follows: We begin in Section~\ref{sect: preli} with basic
definitions and other preliminaries on metric spaces, Hausdorff contents, and Hardy inequalities.
Section~\ref{sect: lemma} is devoted to the statement and proof of our main lemma,
which is then used in Section~\ref{sect: proofs} to prove the results on Hardy inequalities. 
For notation we remark that throughout the paper the letter $C$ is used to denote 
positive constants whose value may change from expression to expression.

\section{Preliminaries}\label{sect: preli}

\subsection{Metric spaces}\label{sect: mms}
We assume that $X=(X,d,\mu)$ is a complete metric measure space
equipped with a metric $d$ and a Borel regular outer measure $\mu$
such that $0<\mu(B)<\infty$ for all balls $B=B(x,r)=\{y\in X:d(x,y)<r\}$.
For $0<t<\infty$, we write $tB=B(x,tr)$, and $\ol B$ is the corresponding
closed ball.
When $A\subset X$,
$\bdry A$ is the boundary and $\ol A$ the closure 
of $A$. 
The distance from $x\in X$ to $A\sub X$ 
is denoted $d(x,A)$. 
When $\Omega\sub  X$ is an open set and $x\in\Omega$, we also denote
$\dom(x)=d(x,\bdry\Omega)$. 

We assume that the measure $\mu$ is {\it doubling}, i.e.\ that there
exists a constant $C_d\ge1$ 
such that 
\[
\mu(2B)\le C_d\,\mu(B)
\]
for all balls $B$ of $X$. The doubling condition together
with the completeness implies that the space $X$ is proper, that is, closed
balls of $X$ are compact. 

The doubling condition gives an upper bound for the dimension of $X$ in the sense
that there is a constant $C=C(C_d)>0$ such that, for $s=\log_2 C_d$,
\begin{equation}\label{doubling dimension}
\frac{\mu(B(y,r))}{\mu(B(x,R))}\ge C\Bigl(\frac rR\Bigr)^s
\end{equation}
whenever $0<r\le R<\diam X$ and $y\in B(x,R)$. The
infimum of the exponents $s$ for which \eqref{doubling dimension} 
holds is called {\it the doubling dimension} of $X$.

Another crucial assumption is 
that the space $X$ supports a {\it (weak) $(1,p)$-Poincar\'e
inequality}. More precisely, we assume that there exist constants $C>0$ and
$\tau\ge1$ such that for all balls $B\sub X$, all continuous
functions $u$, and for all \emph{upper gradients} $g_u$ of $u$, we have the inequality
\begin{equation}\label{poinc}
\vint{B}|u-u_B|\, d\mu
\le C r\Bigl(\;\vint{\tau B}g_u^p\,d\mu\Bigr)^{1/p},
\end{equation}
where 
\[
 u_B=\vint{B}u\,d\mu={\mu(B)}^{-1}\int_Bu\,d\mu
\]
is the integral average of $u$ over $B$.
Recall that a Borel function $g\ge0$ is said to be an upper 
gradient of a function $u$ (on an open set $\Omega\subset X$), 
if for all curves
$\gamma$ joining points $x$ and $y$ (in $\Omega$) we have
\begin{equation}\label{ug}
|u(x)-u(y)|\le\int_{\gamma}g\,ds
\end{equation}
whenever both $u(x)$ and $u(y)$ are finite, and
$\int_{\gamma}g\,ds=\infty$ otherwise. 
By a curve we simply mean a nonconstant, rectifiable, continuous
mapping from a compact interval to $X$. 

Examples of metric spaces satisfying the above
conditions include (weigh\-ted) Euclidean spaces, compact Riemannian manifolds, 
Carnot groups, and metric graphs.
See for instance
\cite{Hj2}, \cite{HjK}, \cite{HEI}, and the references therein for 
more information on analysis on metric spaces based on upper gradients 
and Poincar\'e inequalities.

For the rest of the paper we explicitly assume that in the context of
the $(p,\beta)$-Hardy inequality the space $X$ supports
a $(1,p)$-Poincar\'e inequality. However, in our proofs we sometimes need
to use a $(1,q)$-Poincar\'e inequality for an exponent $q<p$, but
this is justified by the self-improvement property of
Poincar\'e inequalities, see \cite{KeZ}.

Recall that a function $u\colon \Omega\to \R$ is said to be ($L$-)Lipschitz, if
\[
|u(x)-u(y)|\leq L d(x,y)\qquad \text{ for all } x,y\in \Omega.
\]
The set of all Lipschitz functions $u\colon \Omega\to\R$
is denoted $\Lip(\Omega)$, 
and $\Lip_0(\Omega)$ is
the set of Lipschitz functions $u\in\Lip(\Omega)$ with a compact support in $\Omega$.
Recall that the support of a function $u\colon \Omega\to \R$ 
is the closure of the set where $u$ is non-zero.
It is straight-forward to check that the pointwise Lipschitz constant
\[
\Lip(u;x)=\limsup_{y\to x} \frac{|u(x)-u(y)|}{d(x,y)}
\]
defines an upper gradient $g$ for a Lipschitz function
$u\colon \Omega\to \R$ by $g(x)=\Lip(u;x)$.

\subsection{Hausdorff contents}

We measure the thickness of sets $E\sub X$ 
by means of \emph{Hausdorff contents}.
The usual \emph{$\lambda$-Hausdorff content} 
of a set $A \sub X$ is defined by
\[
\Ha^\lambda_\infty(A)=\inf\bigg\{\sum_{k=1}^{\infty}r_k^\lambda :
 A\sub\bigcup_{k=1}^\infty B(x_k,r_k),\ x_k\in A \bigg\},
\]
and the \emph{Hausdorff dimension} of $A$ is then
\[
\dim(A)=\inf\{\lambda>0 : \Ha^\lambda_\infty(A)=0\}.
\]
However, in general metric spaces it is often more convenient to
use a modified version of $\Ha^\lambda_\infty$, namely
the {\it Hausdorff content of codimension $t$}, 
which is given for a set $E\subset X$ by  
\[
\Hh^t_R(E)=\inf\bigg\{\sum_{i\in I} \mu(B(x_i,r_i))\,r_i^{-t} :
E\subset\bigcup_{i\in I} B(x_i,r_i),\ r_i\leq R \bigg\}.
\]
Here we may again assume that $x_i\in E$, as this increases $\Hh^t_R(E)$
at most by a constant factor.

A metric space $X$ is said to be (Ahlfors) $Q$-reqular
if there are constants $c_1,c_2>0$ such that
\[
c_1r^Q\le\mu(B(x,r))\le c_2r^Q
\]
for all balls $B(x,r)$ in $X$.
It is easy to see that in a $Q$-regular space $X$ the content $\Hh^t_\infty(E)$ is comparable
with the usual Hausdorff content $\Ha^{Q-t}_{\infty}(E)$ for every $E\sub X$.

\subsection{Chain condition}\label{sect: chains}

We introduce an important chain condition following \cite{HjK1}; see also \cite{HEI}.
Let $x\in\Omega\sub X$, where $\Omega$ is an open set, and take $\lambda,M\geq 1$ and $a>1$. 
We say that $w\in\ol\Omega$ is connected
to $x$ by a $(\lambda,M,a)$-chain in $\Omega$, denoted $w\in\Ch_\Omega(\lambda,M,a;x)$,
if there exists a sequence of balls $B_k=B(x_k,r_k)$, $k=0,1,2,\dots$, 
so that $x_0=x$ and $x_k\to w$ as $k\to\infty$, and
the following conditions hold for each $k=0,1,2,\dots$:
\begin{enumerate}
 \item[(i)] $\lambda B_k\sub\Omega$\,;
 \item[(ii)] $M^{-1} \dom(x) a^{-k}\leq r_k\leq M \dom(x) a^{-k}$\,;
 \item[(iii)] there is a ball $B_k'$ so that $B_k'\sub B_k\cap B_{k+1}\sub MB_k'$\,.
\end{enumerate}

For instance, if $\Omega\sub X$ is a $C_J$-John domain with center point $x$, and $\lambda\geq 1$, then
there exists $M\geq 1$, depending on $\lambda$, $C_J$, and the doubling constant such that
$w\in\Ch_\Omega(\lambda,M,2;x)$ for each $w\in\ol\Omega$ (see \cite[Thm 9.3]{HjK}).
We mention that in $\R^n$ the sets $\Ch_\Omega(\lambda,M,2;x)\cap\bdry\Omega$
agree with suitable $c$-visual boundaries near $x$, defined in \cite{kole},
where $c$, $\lambda$, and $M$ only depend on each other and $n$,
and we may actually choose $\lambda$ to be as large as we want.

\subsection{Hardy inequalities}\label{sect: hardy}

In the setting of a general metric space the $(p,\beta)$-Hardy inequality takes the following form:
\begin{equation}\label{eq: m-hardy}
\int_{\Omega} |u|^p\, \dom^{\beta-p}\,d\mu
   \leq C\int_{\Omega} g_u^p\, \dom^\beta \, d\mu,
\end{equation}
where $g_u$ is an upper gradient of $u$. We say that
an open set $\Omega\sub X$ admits the (metric) $(p,\beta)$-Hardy inequality if
there exists a constant $C>0$ so that
\eqref{eq: m-hardy} holds for every $u\in\Lip_0(\Omega)$
and for all upper gradients $g_u$ of $u$.

Following the unweighted considerations by Haj\l asz \cite{haj} and 
Kinnunen and Martio \cite{KM}, 
a pointwise version of the weighted $(p,\beta)$-Hardy inequality \eqref{eq: Hardy}
was introduced in \cite{kole}. 
The metric space version reads as follows:
\begin{equation}\label{eq: m w pw Hardy}
|u(x)| \leq C \dom(x)^{1-\frac \beta p} 
     \Big(M_{L\dom(x)}\big(g_u^q{\dom}^{\frac \beta p\,q}\big)(x)\Big)^{1/q}, 
\end{equation}
where $1<q<p$, $L\geq 1$,
and $M_{R}$ 
is the usual restricted Hardy--Littlewood maximal operator,
defined by
\[
M_R f(x)=\sup_{0<r\leq R}\vint{B(x,r)} |f(y)|\,d\mu
\]
for $f\in L^1_{\loc}(X)$.
We say that an open set $\Omega\sub X$ admits the 
pointwise $(p,\beta)$-Hardy inequality if there 
exist some $1<q<p$ and constants $C>0$, $L\geq 1$ so
that the inequality \eqref{eq: m w pw Hardy} holds for all
$u\in\Lip_0(\Omega)$ with these $q$, $C$, and $L$.

\begin{remark}
Using the maximal theorem, it is easy to see that if the pointwise inequality
\eqref{eq: m w pw Hardy} holds for a function $u$ at (almost) every $x\in \Omega$,
then the usual $(p,\beta)$-Hardy inequality holds for $u$ with a constant
only depending on $p$ and the constants from the pointwise inequality and the maximal function
inequality (cf.\ \cite{kole}).
\end{remark}

\section{Main Lemma}\label{sect: lemma}

The proofs of
our main results are based on the following local
estimates for Lipschitz functions vanishing at the boundary. For weight exponents
$\beta\leq 0$ the estimate involves the whole boundary near a point $x\in\Omega$,
whereas for $\beta>0$ we need to restrict to the part of the boundary that
we can connect to $x\in\Omega$ with good chains of balls (cf.\ Section \ref{sect: chains}). 

\begin{lemma}\label{lemma: key estimate}
Let $1<p<\infty$ and $\beta<p$,
let $\Omega\sub X$ be an open set,
and take $x\in\Omega$. Denote $B(x)=\ol B(x,\dom(x))$,
and define
\begin{enumerate}
\item[{(a)}] $E=\bdry\Omega\cap 2B(x)$ if $\beta\leq 0$, 
\item[{(b)}] $E=\Ch_\Omega(\lambda,M,a;x)\cap\bdry\Omega$ if $0<\beta<p$, where $\lambda\geq 2\tau$
($\tau$ is from inequality~\eqref{poinc}). 
\end{enumerate}
Then, for each $0\leq t < p-\beta$, 
there exist an exponent $1<q<p$ and constants $C>0$, $L\geq 1$,
all independent of $x$, such that the estimate 
\begin{equation}\label{eq: bound on Ha}
\begin{split}
 \Hh_{\dom(x)}^t (E)\big|u_{(2\tau)^{-1} B(x)}\big|^{q} \leq 
     C  \dom(x)^{q-\beta\frac q p-t} \int_{LB(x)}
       g_u(y)^q\, \dom(y)^{\beta\frac q p}\,d\mu
\end{split}
\end{equation}
holds for every $u\in\Lip_0(\Omega)$.
\end{lemma}

\begin{proof}
Our proof combines elements from the proofs of \cite[Lemma 5.2]{kole}
and \cite[Thm 5.9]{HeKo}.
Let $0\leq t < p-\beta$.
It is easy to check that we can choose $1<q<\infty$ so that
\[ 
\frac{p}{p-\beta}\,t < q < p. 
\] 
Moreover, we may assume that $X$ supports a $(1,q)$-Poincar\'e inequality (cf.\ \cite{KeZ}).
Also denote $\beta'=\frac q p \beta$.
Then $q/p > t/(p-\beta)$, and we have
\begin{equation}\label{eq: positive}
q-\beta'-t = \tfrac q p (p-\beta) - t > 0.
\end{equation}
Denote $R=\dom(x)$ and $B=\ol B(x,R)$.
If $u_{(2\tau)^{-1}B}=0$ the claim \eqref{eq: bound on Ha} is trivial, so we may assume that $|u_{(2\tau)^{-1}B}|>0$,
and in fact, by homogeneity, that $|u_{(2\tau)^{-1}B}|=1$.
It is also clear that we may assume $E\neq \emptyset$.

\vskip6pt\noindent
{\it Part} (b):
Let us start with the more complicated part (b). The proof of part (a)
goes along the same lines; we comment on the
differences at the end of the proof.

First notice that, by the properties of the chains (Section \ref{sect: chains}), there exists
$L_0\geq 1$, independent of $x$, such that $\Ch_\Omega(\lambda,M,a;x)\sub L_0 B$. 
Now fix $w\in \Ch_\Omega(\lambda,M,a;x)\cap\bdry\Omega$
and let $B_i=B(x_i,r_i)$ be the corresponding chain of balls.
Then
\[\begin{split}
 1\leq |u(w)-u_{(2\tau)^{-1}B}|\leq |u_{B_0}|+|u_{B_0}-u_{(2\tau)^{-1}B}|,
\end{split}
\]
and it follows from the properties of the chain and the assumption $2\tau\leq\lambda$
that $B_0\sub (2\tau)^{-1}B$.
If $|u_{B_0}|<1/2$, we infer, using the above facts and
the $(1,q)$-Poincar\'e inequality, 
that
\begin{equation}\label{eq: if u_B small}\begin{split}
 \tfrac 1 2 & \leq |u_{B_0}-u_{(2\tau)^{-1}B}|
     \leq C R \bigg(\, \vint{(1/2)B} g_u(y)^q\,d\mu\bigg)^{1/q}\\
   & \leq C R^{1-\beta/p} \bigg(\, \vint{(1/2)B} 
        g_u(y)^q \dom(y)^{\beta\frac q p} \,d\mu\bigg)^{1/q}.
\end{split}\end{equation}
As $\Hh_R^t(E)\leq C \mu(B)R^{-t}$,
and $|u_{(2\tau)^{-1}B}|=1$, the claim \eqref{eq: bound on Ha} easily follows
from the doubling condition.

We may hence assume that $1/2 \leq |u_{B_0}|=|u(w)-u_{B_0}|$
for every $w\in \Ch_\Omega(\lambda,M,a;x)\cap\bdry\Omega$.
Using the properties of the chain together with the $(1,q)$-Poincar\'e inequality
and the doubling condition,
and the assumption that the support of $u$ is compact,
we get the standard estimate (see for example \cite{HjK1}) 
\begin{equation}\label{eq: chain}
 1\leq C \sum_{k=0}^\infty r_k\, \Bigl(\;\vint{\tau B_k}g_u^q\,d\mu\Bigr)^{1/q}.
\end{equation}
From \eqref{eq: chain} it follows that there must be a constant
$C_1>0$, independent of $x$, $u$, and $w$, and at least one index $k_w\in\N$
so that
\begin{equation}\label{eq: one big}
r_{k_w}\, \Bigl(\;\vint{\tau B_{k_w}}g_u^q\,d\mu\Bigr)^{1/q} \geq 
   C_1 a^{-k_w\alpha} = C_1 R^{-\alpha} {r_{k_w}}^\alpha,
\end{equation}
where we choose $\alpha=\frac 1 q (q-\beta'-t)>0$ (by \eqref{eq: positive}).
Let us write from now on
$B_w=B(x_w,r_w)$ instead of $B_{k_w}=B\big(x_{k_w},r_{k_w}\big)$.

We assumed that $\tau B_w\sub(\lambda/2) B_w$, and as
$\lambda B_w\sub\Omega$, it follows that ${r_{w}}^\beta\leq C \dom(y)^\beta$ 
for each $y\in \tau B_{w}$. Thus
\begin{equation}\label{eq: holder out}\begin{split}
 \Bigl(\;\vint{\tau B_{w}}g_u^q\,d\mu\Bigr)^{1/q}
  \leq C r_{w}^{-\beta/p}  \mu(\tau B_w)^{-1/q} \bigg(\int_{\tau B_{w}}g_u(y)^q 
           \dom(y)^{\beta\frac q p}\,d\mu\bigg)^{1/q}.
\end{split}\end{equation}
In particular, combining \eqref{eq: one big} and \eqref{eq: holder out}
we obtain for each $w\in E$ a ball $B_w$
such that
\begin{equation}\label{eq: pre good r}
 \mu(\tau B_w)^{1/q}\, r_w^{\alpha - 1 + \beta/p }  \leq C R^\alpha
 \bigg(\int_{\tau B_w}g_u(y)^q 
           \dom(y)^{\beta'}\,d\mu\bigg)^{1/q}.
\end{equation}
But here $\alpha - 1 + \beta/p = t/q$, so by raising both sides 
of \eqref{eq: pre good r} to
power $q$ we get a useful estimate
\begin{equation}\label{eq: good r}
 \mu(\tau B_w)\,r_w^{-t}\leq C R^{q-\beta' - t}
   \int_{\tau B_w}g_u(y)^q 
           \dom(y)^{\beta'}\,d\mu.
\end{equation}

Using again the properties of the chain we see that there exists
$\tau'\geq\tau$ such that 
$\tau B_w\sub B(w,\tau'r_w)$ holds for all $w\in E$.
By the standard $5r$-covering lemma (see e.g.\ \cite{HEI}), there
exist points $w_1,w_2,\ldots\in E$
so that
if we denote $r_i=\tau' r_{w_i}$, then the balls
$\tilde B_i=B(w_i,r_i)$ are pairwise disjoint, but still
$E \sub \bigcup_{i=1}^\infty 5\tilde B_i$. 
Moreover, it is easy to find $L\geq 1$, independent of
$x$, so that $\tilde B_i\sub L B$ for all $i$; recall that $B=\ol B(x,\dom(x))$.
Estimate \eqref{eq: good r}, the doubling property,
and the pairwise disjointness of the balls $\tau B_{w_i}\sub \tilde B_i\sub LB$ 
immediately yield
\begin{equation}\label{eq: Hastimate}
\begin{split}
\Hh^t_R(E)
& \leq \sum_{i=1}^\infty \mu(5\tilde B_i)(5r_i)^{-t} \leq C \sum_{i=1}^\infty \mu\big(\tau B_{w_i}\big)r_{w_i}^{-t}\\
& \leq \sum_{i=1}^\infty C R^{q-\beta' - t}
   \int_{\tau B_{w_i}}g_u(y)^q 
           \dom(y)^{\beta'}\,d\mu\\
& \leq C R^{q-\beta' - t}
   \int_{LB}g_u(y)^q 
           \dom(y)^{\beta'}\,d\mu.
\end{split}
\end{equation}
As we assumed $|u_{(2\tau)^{-1} B}|=1$ and denoted $\beta'=\beta\frac q p$,
estimate \eqref{eq: bound on Ha} for part (b) is proven.

\vskip6pt\noindent
{\it Part}  (a):
Let us only give here a brief description of the main differences 
compared to part (b).
We begin by fixing $w\in E=\bdry\Omega\cap 2B$, then
define $r_k=2^{-k} R$, $k\in\N$, and 
denote $B_k=B(w,r_k)$. 

If $|u_{B_0}|<1/2$, we see with a calculation similar to
\eqref{eq: if u_B small} and using the inclusion
$B_0\sub 3B$ that
\begin{equation*}\label{eq: if u_B small*}\begin{split}
 1 \leq C R^{1-\beta/p} \bigg(\, \vint{3\tau B} 
        g_u(y)^q \dom(y)^{\beta\frac q p} \,d\mu\bigg)^{1/q}.
\end{split}\end{equation*}
Notice that the assumption $\beta\leq 0$ guarantees that $R^\beta\leq C \dom(y)^\beta$ 
for each $y\in 3\tau B$. The claim \eqref{eq: bound on Ha} follows.

We may hence assume that $1/2 \leq |u_{B_0}|=|u(w)-u_{B_0}|$.
But now estimate \eqref{eq: chain} follows again for balls $B_k$ by a standard
`telescoping' argument 
using the $(1,q)$-Poincar\'e inequality (cf.\  e.g.\  \cite{HEI}), 
and the rest of the proof is almost identical to part (b);
for instance, \eqref{eq: one big} holds now with $a=2$.
Notice in particular that since 
$\beta\leq 0$ and $w\in E$, it follows again that ${r_{k}}^\beta\leq \dom(y)^\beta$ 
for each $y\in B_{k}$. At the end we can use the $5r$-covering theorem 
directly to balls $\tau B_{w}$, as they are now centered at $w$,
and the desired Hausdorff content estimate follows just as in \eqref{eq: Hastimate}.
The proof is complete. 
\end{proof}

\section{Weighted and pointwise inequalities}\label{sect: proofs}

Let us begin this section by rephrasing Theorem \ref{thm: main}
in general metric spaces:

\begin{thm}\label{thm: main mms}
Let $\Omega\sub X$ be an open set and 
let $1<p<\infty$. 
Suppose that there exist
an exponent $t\geq 1$ and a constant $C_0>0$ such that
\begin{equation}\label{eq: main estimate}
 \Hh_{\dom(x)}^t\big(\bdry\Omega\cap \ol B(x,2\dom(x))\big)
    \geq C_0 \mu\big(\ol B(x,\dom(x))\big) \dom(x)^{-t}
\end{equation}
for all $x\in\Omega$.
Then $\Omega$ admits the $(p,\beta)$-Hardy inequality
for all $\beta<p-t$.
\end{thm}

In the case $\beta\leq 0$, Theorem~\ref{thm: main mms} is 
an immediate consequence of the following result
on pointwise inequalities. The special case $\beta=0$
of Theorem~\ref{thm: pw for beta<0} is contained in the
results of \cite{KLT}.

\begin{thm}\label{thm: pw for beta<0}
Let $\Omega\sub X$ be an open set and 
let $1<p<\infty$ and $\beta\leq 0$.
Suppose that there exist
an exponent $0\leq t < p - \beta$ and a constant $C_0>0$ such that
the density condition \eqref{eq: main estimate} holds
for all $x\in\Omega$.
Then $\Omega$ admits the pointwise $(p,\beta)$-Hardy inequality.
%provided that $\beta<p-t$.
\end{thm}

\begin{proof}
Let $u\in\Lip_0(\Omega)$, $x\in\Omega$,
and denote $R=\dom(x)$,
$B=B\big(x,(2\tau)^{-1}R\big)$. Then
\[
|u(x)|\le|u(x)-u_{B}|+|u_{B}|.
\]
Choose $1<q<p$ just as in the proof of Lemma \ref{lemma: key estimate}.
Again a standard telescoping trick 
using the $(1,q)$-Poincar\'e inequality gives 
\[\begin{split}
|u(x)-u_{B}| & \leq CR\big( M_{R/2} g_u^q(x)\big)^{1/q}\\
 & \leq CR^{1-\beta/p}\big( M_{R/2} (g_u^q \dom^{\beta q/p})(x)\big)^{1/q}.
\end{split}
\]  
On the other hand, using assumption \eqref{eq: main estimate}, 
part (a) of Lemma \ref{lemma: key estimate}, and the doubling condition, we obtain
\[\begin{split}
|u_B|^q & \leq C
\mu\big(\ol B(x,R)\big)^{-1} R^{t}
      R^{q-\beta\frac q p-t} \int_{B(x,LR)}
       g_u(y)^q\, \dom(y)^{\beta\frac q p}\,d\mu\\
 & \leq  C R^{q-\beta\frac q p} M_{LR}\big(g_u^q \dom^{\beta\frac q p}\big)(x).
\end{split}\]
The pointwise $(p,\beta)$-Hardy inequality follows
easily from the previous estimates.
\end{proof}

Conversely, we have the following necessary condition for pointwise Hardy inequalities.
The Euclidean case was proven in \cite{lene}, and the special case $\beta=0$ 
was done in \cite{KLT} in the metric space setting. We omit the proof here, as the
modifications needed to the proofs in \cite{lene} and \cite{KLT} are obvious.

\begin{prop}\label{prop: nec}
Let $\Omega\sub X$ be an open set and 
assume that $\Omega$ admits the pointwise $(p,\beta)$-Hardy inequality
\eqref{eq: m w pw Hardy}. If $\beta<0$, we assume in addition that
$X$ is $Q$-regular.
Then there exists $t<p-\beta$ such that
the inner density condition \eqref{eq: main estimate} holds
for all $x\in\Omega$.
\end{prop}

Combining Theorem \ref{thm: pw for beta<0} and Proposition \ref{prop: nec}
we obtain a characterization in the case $\beta < 0$, which is new even in $\R^n$;
for $\beta=0$ the corresponding result holds in metric spaces by \cite{KLT}.

\begin{coro}\label{coro: char}
Assume that $X$ is $Q$-regular. Let $\Omega\sub X$ be an open set and 
let $1<p<\infty$ and $\beta < 0$.
Then $\Omega$ admits the pointwise $(p,\beta)$-Hardy inequality \eqref{eq: m w pw Hardy}
if and only if there exist
an exponent $t < p - \beta$ and a constant $C_0>0$ such that
the inner density condition \eqref{eq: main estimate} holds
for all $x\in\Omega$.
\end{coro}

Let us finally give a proof for our main result concerning weighted Hardy inequalities:

\begin{proof}[Proof of Theorem \ref{thm: main mms}]
As the pointwise $(p,\beta)$-Hardy inequality
always implies the usual $(p,\beta)$-Hardy inequality 
(see the remark at the end of Section \ref{sect: hardy}),
Theorem \ref{thm: main mms} follows from Theorem \ref{thm: pw for beta<0}
for $\beta\leq 0$.

Hence, we only need to consider the case $0<\beta<p-t$.
But now $p-\beta > t \geq 1$, and so
Theorem \ref{thm: pw for beta<0}, applied to the unweighted case, implies that $\Omega$
admits the $(p-\beta,0)$-Hardy inequality. 
The claim now follows in fact from a 
straight-forward metric space generalization of a
more general result of \cite[Lemma 2.1]{lesi}, 
but let us recall here the calculations in this special case
for the sake of completeness and in order to
emphasize the role of elementary tools behind the theorem.

To this end, let $u\in \Lip_0(\Omega)$
and let $g_u$ be an upper gradient of $u$. Now define $v=|u|^{\frac{p}{p-\beta}}$. 
As $\beta>0$, we see that
$v$ is a Lipschitz-function with a compact support in $\Omega$, and, moreover,
\[
g_v(x) = \big(\tfrac{p}{p-\beta}\big)|u(x)|^{\beta/(p-\beta)}g_u(x) 
\]
defines an upper gradient for $v$. 
As the $(p-\beta,0)$-Hardy inequality holds for
$v$, we obtain, with the help of H\"older's inequality
(observe $\frac{p-\beta}{p}+\frac\beta p=1$), that
\begin{equation}\label{eq: hardy for v}\begin{split}
 \int_\Omega  |u(x)&|^{p}  \dom(x)^{-(p-\beta)}\,d\mu
  \ = \ \int_\Omega |v(x)|^{p-\beta}  \dom(x)^{-(p-\beta)}\,d\mu\\
  &  \leq C_1 \int_\Omega g_v(x)^{p-\beta}\,d\mu
     \ = \  C_2 \int_\Omega |u(x)|^\beta g_u(x)^{p-\beta}\,d\mu\\
  &  = C_2 \int_\Omega \Big(|u(x)|^\beta\dom(x)^{\frac{\beta(\beta-p)}{p}}\Big)
       \Big(g_u(x)^{p-\beta}\dom(x)^{\frac{\beta(p-\beta)}{p}}  \Big)\,d\mu\\
  & \leq C_2  
  \bigg(\int_\Omega |u(x)|^p\dom(x)^{\beta-p}\,d\mu \bigg)^\frac{\beta}{p}
     \bigg(\int_\Omega g_u(x)^p\dom(x)^{\beta}\,d\mu\bigg)^{\frac{p-\beta}{p}}.
\end{split}
\end{equation}
From \eqref{eq: hardy for v} the $(p,\beta)$-Hardy inequality for $u$ easily follows 
by first dividing with the first integral term on the right-hand side (which we may assume to be
non-zero), and then taking both sides to power $p/(p-\beta)$. 
\end{proof}

\begin{remark}
There is an interesting observation concerning the procedure in \eqref{eq: hardy for v} 
and the best possible constants in Hardy inequalities. Namely, 
it is well-known that the best possible constant for the $p$-Hardy inequality
in a convex domain $\Omega\sub\R^n$, $n\geq 1$, is $C_1=(p/(p-1))^p$,
and in other smooth domains the constant is in general larger;
see e.g.\ \cite{HLP} ($n=1$) and \cite{MMP} ($n\geq 2$).
On the other hand, the best possible constant for the one-dimensional 
weighted $(p,\beta)$-Hardy inequality,
with $\beta<p-1$, is $(p/(p-\beta-1))^p$ (see \cite{HLP}), 
and this can be trivially generalized to, say, a ball or a half-space in $\R^n$.

Usually, our methods on Hardy inequalities lead to constants which are far from being optimal.
However, if we have in \eqref{eq: hardy for v} that
$
C_1=\big(\frac {p-\beta}{p-\beta-1}\big)^{p-\beta} 
$
(the optimal constant for the $(p-\beta)$-Hardy), we see directly from the calculation that
the constant in the corresponding $(p,\beta)$-Hardy inequality will then be
$
C=\big(\frac {p}{p-\beta-1}\big)^{p}, 
$
which is, at least in the above special cases, the best possible constant
for the $(p,\beta)$-Hardy inequality.
This raises the question whether the above procedure preserves 
the optimal constants in (weighted) Hardy inequalities even in more general cases. 
\end{remark}

We finish by recording the following result, which
coincides in the case $X=\R^n$ with the main theorem of \cite{kole}. The proof 
is identical to the proof of Theorem \ref{thm: pw for beta<0} besides that
the part (b) of Lemma \ref{lemma: key estimate} (instead of part (a)) is needed.
When $X=\R^n$,
the derivation of this result using Lemma \ref{lemma: key estimate}(b) 
is --- once all the technicalities arising from the metric space setting are 
removed --- in a theoretical sense a simplification over the proof from \cite{kole}, as here we avoid completely the
use of Whitney type coverings and Frostman's lemma.

\begin{thm}\label{thm: pw for beta>0}
Let $\Omega\sub X$ be an open set and 
let $1<p<\infty$. 
Suppose that there exist
an exponent $t \geq 0$ and constants $C_0>0$, $M\geq 1$, $\lambda\geq 2\tau$, and $a>1$ such that
\begin{equation}\label{eq: main estimate for chains}
 \Hh_{\dom(x)}^t\big(\Ch_\Omega(\lambda,M,a;x)\cap\bdry\Omega\big)
    \geq C_0 \mu\big(\ol B(x,\dom(x))\big) \dom(x)^{-t}
\end{equation}
for all $x\in\Omega$.
Then $\Omega$ admits the pointwise $(p,\beta)$-Hardy inequality
for all $\beta<p-t$, and hence also the usual $(p,\beta)$-Hardy inequality
for these $\beta$.
\end{thm}

Notice that, contrary to Theorem \ref{thm: main mms}, we do not need to
assume above that $t\geq 1$, only $t\geq 0$. Moreover, as noted in the Introduction,
examples from \cite{kole} show that some kind of an `accessibility' 
condition is needed to guarantee even the validity of the usual
integral Hardy inequalities
when $\beta\geq p-1$ (corresponding to the case $0\leq t<1$).
As a concrete example 
we mention that if $\Omega\sub\R^n$  is a uniform domain
satisfying
\eqref{eq: main estimate} for all $x\in\Omega$, then also 
\eqref{eq: main estimate for chains} holds
with suitable constants for all $x\in\Omega$;
see \cite{kole} for details. In particular, this is true in a snowflake domain,
where we can have $t<1$.

\subsection*{Acknowledgement}
The author is grateful to Professor Pekka Koskela
for helpful comments during the preparation of this work.

\end{document}